\documentclass{siamart220329}

\usepackage{amsmath,amssymb}
\usepackage{epsfig,graphics}
\usepackage{booktabs}
\usepackage{comment}
\usepackage{url}
\usepackage{tikz}
\graphicspath{{./figures/}}

\usepackage{algorithm}
\usepackage{algpseudocode}

\usepackage{array,booktabs}
\usepackage{siunitx}

\usepackage{caption} 

\usepackage{eqparbox}
\newdimen{\algindent}
\algnewcommand{\LeftComment}[1]{\Statex \hspace{\algindent}
\( \triangleright \) #1}

\usepackage{subfigure} 
\usepackage{threeparttable} 
\def\x{{\bf x}}

\def\beq{\begin{equation}}
\def\eeq{\end{equation}}
\def\eeq{\end{equation}}

\def\erf{{\rm erf}}

\def\rprime{\hbox{\hskip.25em\raise.5ex\hbox{$'$}\hskip.15em}}

\def\ddn1{{\frac{\partial}{\partial \nu_{\yb}}}}

\newcommand{\ignore}[1]{}

\newtheorem{remark}[theorem]{Remark}

\newcommand{\be}{\begin{equation}}
\newcommand{\ee}{\end{equation}}
\newcommand{\ba}{\begin{aligned}}
\newcommand{\ea}{\end{aligned}}
\newcommand{\bea}{\begin{eqnarray}}
\newcommand{\eea}{\end{eqnarray}}

\def\ymsc{Yau Mathematical Sciences Center,
  Tsinghua University, Beijing China 100084}

\def\maththu{Department of Mathematics,
  Tsinghua University, Beijing China 100084}

\def\qiuzhen{Qiuzhen College, 
  Tsinghua University, Beijing China 100084}

\def\papertitle{An integral equation method for linear two-point boundary value systems}

\title{\papertitle}
\author{
  Tianze Zhang\thanks{\qiuzhen\,
    ({\tt zhangtz21@mails.tsinghua.edu.cn}).} \and
 Yixuan Ma\thanks{\maththu\,
    ({\tt mayx22@mails.tsinghua.edu.cn}).} \and
 Jun Wang\thanks{\ymsc\,
    ({\tt jwang2020@tsinghua.edu.cn, corresponding author}).} 
}

\begin{document}

\maketitle

\begin{abstract}
We present an integral equation-based method for the numerical solution of
two-point boundary value systems. Special care is devoted to the mathematical
formulation, namely the choice of the background Green's function
that leads to a well-conditioned integral equation. We then make use of
a high-order Nystr\"om discretization and a fast direct solver on the continuous level
to obtain a black-box solver that is fast and accurate. 
A numerical study of conditioning of different integral formulations is carried out. Excellent performance in speed, accuracy, and robustness is demonstrated with several
challenging numerical examples.
\end{abstract}



\pagestyle{myheadings}
\thispagestyle{plain}
\markboth{T. Zhang, Y. Ma, J. Wang}
{An integral equation method for two-point BVP systems}

\tableofcontents 
\vspace{3mm}
\section{Introduction \label{sec:intro}}  

Many problems in science and engineering require the numerical solution of two-point boundary value problems for systems of ordinary differential equations, which take the form of
\begin{equation}  \label{bvpnonlin}
\begin{split}
\Phi'(x) &=  F(\Phi(x),x), \\
A\cdot \Phi(a) &+ C\cdot \Phi(c) =\gamma \;, 
\end{split}
\end{equation}
where $\Phi: [a,c]\rightarrow \mathbb{R}^n$ is in $C^1[a,c]$, and $F: \mathbb{R}^{n+1}\rightarrow \mathbb{R}^n$ is continuous. $A, C\in\mathbf{L}(\mathbb{R}^{n\times n})$, and $\gamma \in \mathbb{R}^n$.

For such a nonlinear problem, a Newton-like iterative method is usually required, reducing each step
of iteration to a linear problem of the form
\begin{equation} \label{bvplin}
\begin{split}
\Phi'(x) + p(x)\cdot \Phi(x) &= f(x) \\
A\cdot \Phi(a) + C\cdot \Phi(c) &=\gamma \;,
\end{split}
\end{equation}
where $p: [a,c]\rightarrow \mathbf{L}(\mathbb{R}^{n\times n})$ and $f: [a,c]\rightarrow \mathbb{R}^q$
are continuous.

In this paper, we focus primarily on the linear problem \eqref{bvplin} and postpone the discussion of
the nonlinear case \eqref{bvpnonlin} to a later date. We are especially interested in the so-called
singularly perturbed problems where a small parameter appears on the leading-order term, rendering complicated local structures even for the linear case. We make the further observation that the
small parameter often arises when the method of continuation (or homotopy method) is applied for finding a suitable initial guess for the solution of nonlinear problems. It is referred to in the literature as ``introducing different time scales". As a consequence, proper treatment of the singular perturbation
is essential for linear and nonlinear problems alike.

There exists an abundance of numerical methods for the problem \eqref{bvplin} such as finite difference methods, finite element methods, and collocation methods \cite{tpbvp_book,colloc_ascher,bader,coutsias, gottlieb_orszag,keller,kreiss1972,kreiss1981,kreiss1986,lentini1977,mattheij1984}. We do not seek to review the whole subject here, but make the observation that all the above-mentioned methods are
based on the discretization of the differential operator, producing a condition number that approaches
infinity as the computational grid is refined. 

Considering the above-mentioned problems, integral equation methods \cite{greengardrokhlin,leegreengard,starrrokhlin,Greengard1991,leebrokhlin2020} become an excellent choice for this situation. The key step is to reformulate the boundary value problem into a second-kind Fredholm integral equation, which can be stably discretized, with a condition number that approaches the conditioning of the physical problem itself. All the nice properties of the numerical method follow naturally. In \cite{greengardrokhlin}, a fast algorithm for the second-order scalar equation is designed, exploiting the hierarchical low-rank structure of the underlying linear system. In \cite{leegreengard}, a robust a posteriori error estimate for the integral equation is presented, leading to a reliable strategy for adaptivity. In \cite{starrrokhlin}, the mathematical formulation and fast algorithm for the system's case \eqref{bvplin} is considered. This is by far the only paper we can locate that discusses the integral equation method for systems of two-point boundary value problems.

Unfortunately, the method proposed in \cite{starrrokhlin}, though fast and accurate, has not received wide applications
in practice. There are two reasons for this: (i) The method is nonadaptive. (ii) The method requires an additional condition on the boundary condition matrices:
\begin{equation} \label{bccond}
det(A+C)\neq 0\;.
\end{equation}
We make the remark that condition \eqref{bccond} is not the condition for the existence or uniqueness of the original problem \eqref{bvplin}. It is purely a requirement for the successful construction of the so-called background Green's function utilized in the numerical algorithm. There exist plenty of problems where the condition is violated (which is sometimes referred to as ``degenerate BC"), but the solution exists and is unique. (This will be discussed in detail in Section 2.) As an amendment to problem (ii), \cite{starrrokhlin} states that it is possible to introduce a nonsingular linear transform $T(x): [a,c]\rightarrow \mathbf{L}(\mathbb{R}^{n\times n})$, so that the transformed problem with $\tilde{\Phi}(x)=\Phi(x)T(x)$ has a nondegenerate BC. However, neither a condition for the existence of such transforms nor an explicit construction is given, leaving the user in the situation of having to construct
such a transform manually before calling the solver. A more subtle problem is as follows: When the linear transform is introduced, the conditioning of the problem is changed. An ill-conditioned transform can
introduce ill-conditioning into the picture, destroying one of the nicest features of integral equation methods. 

Among the two problems mentioned above, the first one is easier to overcome and is discussed in \cite{zhang2025}. The main contribution of this paper is to overcome the second problem, which is more about the mathematical formulation than the algorithm. Two different approaches are presented to deal with degenerate boundary conditions, and different conditioning of the different formulations are compared in a case study. A heuristic explanation for the different performance is given while a comprehensive analysis of conditioning is postponed to a later date. The final product combining the framework of \cite{starrrokhlin} and the technologies developed in this paper is a black-box solver for systems of linear two-point boundary value problems \eqref{bvplin}, which is fast (achieving linear scaling in the number of grid points), accurate (with arbitrary order of convergence), and naturally compatible with adaptive refinement and coarsening, justifying the advantages of integral equation methods.

The paper is organized as follows. In Section 2, we discuss the integral equation formulation for systems
of two-point boundary value problems, providing two different approaches for the construction of
background Green's functions. In Section 3, key numerical components are reviewed, including a high-order Nystr\"om discretization~\cite{kress_book} based on Clenshaw-Curtis quadrature~\cite{ccquad,fox1968,atap2019} and a fast direct solver on the continuous level~\cite{fds2019}. These components have become classic numerical methods in the last few decades. In Section 4, a case study is carried out to illustrate the different condition numbers of different formulations. Several challenging numerical examples are presented in Section 5 to demonstrate the performance of our algorithm in terms of speed, accuracy, and robustness. Finally, in Section 6 we summarize what the algorithm achieves and make a few remarks on directions of future work.

\section{Mathematical Preliminaries\label{sec:prelim}}
In this section, we discuss in detail integral equation formulations of systems of two-point boundary value problems. We begin this section by introducing some notation and definitions for the integral equation formulation. 
\subsection{notation and definitions}
\begin{definition}
    When $f(x)=0$ in \eqref{bvplin}, the equation 
    \begin{equation}\label{eqn:homo}
        \Phi(x)+p(x)\cdot\Phi(x) =0 
    \end{equation}
    is referred to as a homogeneous equation, while the general case
    \begin{equation}\label{eqn:inhomo}
        \Phi(x)+p(x)\cdot\Phi(x) = f(x)
    \end{equation}
    is referred to as an inhomogeneous equation.
\end{definition}

Similarly, we have the following definition for boundary conditions.
\begin{definition}
    When $\gamma=0$ in \eqref{bvplin}, the boundary condition 
    \begin{equation}\label{bc:homo}
        A\cdot \Phi(a) + C\cdot \Phi(c) =0 
    \end{equation}
    is referred to as a homogeneous boundary condition, while the general case
    \begin{equation}\label{bc:inhomo}
        A\cdot \Phi(a) + C\cdot \Phi(c) = \gamma
    \end{equation}
    is referred to as an inhomogeneous boundary condition.
\end{definition}

Due to the linearity of the problem, the solution to \eqref{eqn:inhomo} and \eqref{bc:inhomo} can be represented as:
\begin{equation} \label{eqn:sep_sol}
 \Phi(x)=\tilde{\Phi}(x)+\Phi_b(x)\;,   
\end{equation}
in which $\tilde{\Phi}(x)$ solves 
\begin{equation} \label{bvplinhomo}
\begin{split}
\tilde{\Phi}'(x) + p(x)\cdot \tilde{\Phi}(x) &= \tilde{f}(x) \\
A\cdot \tilde{\Phi}(a) + C\cdot \tilde{\Phi}(c) &= 0 \;,
\end{split}
\end{equation}
where $\tilde{f}(x)=f(x)-p(x)\cdot \Phi_b(x)$, while $\Phi_b(x)$ is an arbitrary function of $C^1$ that
satisfies the inhomogeneous boundary condition \eqref{bc:inhomo}.

When $det(A+C)\neq 0$, for example, an obvious choice of $\Phi_b(x)$ is $\Phi_b(x)=(A+C)^{-1}\gamma$.
The cases with $det(A+C)=0$ will be discussed later in this section. 
For the remainder of this paper, we will focus on \eqref{bvplinhomo} and will not distinguish $\tilde{\Phi}(x)$ and $\Phi(x)$, or $\tilde{f}(x)$ and $f(x)$.

A classical tool for analyzing systems of ordinary differential equations is the fundamental matrix, which is defined as follows.
\begin{definition}
    A matrix-valued function $\Gamma(x): [a,c]\rightarrow \mathbf{L}(\mathbb{R}^{n\times n})$ is called a fundamental matrix for the ordinary differential equation \eqref{eqn:homo}.
    if it is nonsingular for all $x\in [a,c]$ and 
     \begin{equation} \label{def:fundamat}
        \Gamma'(x)+p(x)\cdot \Gamma(x) = 0\;,
    \end{equation} 
    for all $x\in [a,c]$.
\end{definition}

The existence and uniqueness of the solution to equation \eqref{eqn:homo} and \eqref{bc:inhomo} can be stated in terms of the fundamental matrix. Proofs can
be found in classical texts on ordinary differential equations~\cite{coddington_levinson,walter_book}.
\begin{theorem}\label{thm:exist_unique}
    Let $\Gamma(x)=(\gamma_1(x), \gamma_2(x),\cdots \gamma_n(x))$ be a fundamental matrix of the differential equation \eqref{eqn:homo}.
    Then the following properties are equivalent.
    \begin{itemize}
        \item The homogeneous boundary value problem \eqref{eqn:homo}
        and \eqref{bc:homo} has only the trivial solution $\Phi(x)\equiv 0$.
        \item The matrix $D=A\cdot \Gamma(a)+C\cdot \Gamma(c)$ is non-singular, that is, $det(D)\neq 0$.
        \item For given $f\in C[a,c]$ and $\gamma \in\mathbb{R}^n$, the boundary value problem \eqref{eqn:inhomo} and \eqref{bc:inhomo}
        has a unique solution. 
    \end{itemize}
\end{theorem}

From now on, we assume that the boundary value problem has a unique solution.
The principal analytical tool that enables the problem \eqref{bvplinhomo} to be solved as a second-kind Fredholm integral equation is the definition of Green's function \cite{starrrokhlin}:

\begin{definition}[Green's function]\label{greensfunc}
    A continuous function $G(x,t): [a,c]\times [a,c]\rightarrow \mathbf{L}(\mathbb{R}^{n\times n})$
    is referred to as the Green's function for the problem \eqref{eqn:homo}, \eqref{bc:homo}, if
    \begin{itemize}
    \item $\frac{\partial G(x,t)}{\partial x}$ is continuous except at $x=t$.
    \item $G(x+0,x)-G(x-0,x)=I_n$ for all $x\in [a,c]$.
    \item $\frac{\partial}{\partial }G(x,t)+p(x)\cdot G(x,t)=0$ for all $x,t\in[a,c],\; x\neq t$.
    \item $A\cdot G(a,t)+C\cdot G(c,t)=0$ for all $t\in [a,c]$.
    \end{itemize}
\end{definition}

\begin{remark}
    If $det(D)\neq 0$, then the Green's function exists and is unique~\cite{bliss1926}.
\end{remark}

\begin{remark}
    If Green's function for \eqref{eqn:homo}, \eqref{bc:homo} is known, the solution to the problem \eqref{eqn:inhomo}, \eqref{bc:homo} can be expressed as a convolution:
    \begin{equation}\label{sol_int}
        \Phi(x)=\int_a^c G(x,t) f(t)\;dt\;.
    \end{equation}
    In other words, the Green's function provides the solution operator to problem \eqref{eqn:inhomo}, \eqref{bc:homo} semi-analytically.
\end{remark}

\subsection{Formulation of integral equations} \label{sec:formulateie}
Explicit expressions of Green's functions are generally unavailable. Fortunately, for simple problems
whose fundamental matrices are known, the Green's function can be constructed explicitly. We refer to such problems and their Green's functions as background problems and background Green's functions, respectively,
and use them to represent the solution to harder problems.

If the fundamental matrix (for a background problem) is known, the existence of the Green's function is determined by the following boundary condition matrix:
\begin{definition}[boundary condition matrix]
    Assume that the fundamental matrix $\Gamma_0(x)$ for a problem 
    \begin{align}
        & \Phi'(x) +p_0(x)\cdot \Phi(x) = 0\;, \label{eqn:homo0}\\\
        & A\cdot \Phi(a) + C\cdot \Phi(c) =0 \label{bc:homo0}
    \end{align}
    is known. Then the matrix
    \begin{equation}\label{eqn:defd0}
        D_0 := A\cdot \Gamma_0(a)+A\cdot \Gamma_0(c)
    \end{equation}
    is referred to as the boundary condition matrix.
\end{definition}

\begin{remark}
    Equations \eqref{eqn:homo0} and \eqref{bc:homo0} have exactly the same form as \eqref{eqn:homo} and \eqref{bc:homo}. But we use the notation $p_0, \Gamma_0, D_0$ to highlight the fact that they can be different from those of the problem we are solving.
\end{remark}

Now the explicit construction of a (background) Green's function can be summarized in the following theorem.
\begin{theorem}[construction of Green's function via fundamental matrix]
    Assume that a (background) problem \eqref{eqn:homo0},\eqref{bc:homo0} has a fundamental matrix $\Gamma_0(x)$ with
$D_0$ defined in \eqref{eqn:defd0} nonsingular, then there exists a unique Green's function $G_0(x,t)$ for this problem.
$G_0(x,t)$ is given by the formula
\begin{equation} \label{eqn:gviagamma}
G_0(x,t)=
\begin{cases}
\Gamma_0(x)\cdot\nu_L(t) \;\; &(t\leq x) \\
\Gamma_0(x)\cdot \nu_R(t) \;\; &(t\geq x) \;,
\end{cases}
\end{equation}
where $\nu_L(t)=\Gamma_0^{-1}(t)+J_0(t)$, $\nu_R(t)=J_0(t)$, and $J_0(t)=-D_0^{-1}\cdot C\cdot \Gamma_0(c)\cdot \Gamma_0^{-1}(t)\;.$
\end{theorem}

Under the same condition, we can also construct the inhomogeneous term $\Phi_b(x)$ for the general case, namely 
\begin{equation}
    \Phi_b(x)=\Gamma_0(x)D_0^{-1}\gamma\;.
\end{equation}

Once a suitable background Green's function is constructed, we can seek a solution $\Phi$ to the original problem
\eqref{eqn:inhomo} \eqref{bc:homo} with the form
\begin{equation} \label{eqn:solrep}
\Phi(x) =\int_a^c G_0(x,t) \sigma(t) \;dt\;,
\end{equation}
where $\sigma\in C[a,c]$ is an unknown function defined on the interval $[a,c]$.
Any function written in the form of \eqref{eqn:solrep} readily satisfies the homogeneous boundary condition \eqref{bc:homo}. To enforce the differential equation, we substitute \eqref{eqn:solrep} into
\eqref{eqn:inhomo} and derive an integral equation for the unknown  function $\sigma$:
\begin{equation} \label{eqn:eqnsigma}
\sigma(x)+\tilde{p}(x)\cdot \int_a^c G_0(x,t)\cdot \sigma(t)\;dt = f(x)\;,
\end{equation}
where $\tilde{p}(x)=p(x)-p_0(x)$.

We now discretize the integral equation \eqref{eqn:eqnsigma} instead of the original boundary
value problem \eqref{eqn:inhomo}, \eqref{bc:homo}. Once the numerical solution of $\sigma$ is obtained,
it is substituted into equation \eqref{eqn:solrep} to recover solution $\Phi(x)$.
These are the key steps of integral equation methods.

\begin{remark}
Equation \eqref{eqn:eqnsigma} assumes the form of a second kind Fredholm equation (assuming
$p, p_0, f\in C([a,c])$). It enjoys a collection of nice properties, including a condition number that approaches that of the original problem. We do not seek to
review the complete subject here. Interested readers are recommended to \cite{kress_book}.
\end{remark}

\subsection{Choices of background Green's functions}
The discussion of the last subsection reveals that the successful design of an algorithm
for equation \eqref{eqn:eqnsigma} depends first and foremost on the choice of a suitable
background Green's function. It should be at least (i) easy to compute (as explicit as possible)
and (ii) well-conditioned (as a matrix, for all $x,t\in[a,c]$). We begin by listing the two options
offered in \cite{starrrokhlin}.

\begin{lemma}
A fundamental matrix $\Gamma_0$ for the equation
\begin{equation} \label{eqn:background1}
\Phi'(x) = 0
\end{equation}
is given by the formula
\begin{equation} \label{eqn:fundmat1}
\Gamma_0(x) = I_n\;,
\end{equation}
where $n$ is the dimensionality of the problem \eqref{eqn:inhomo}, \eqref{bc:homo},
and $I_n$ denotes the unity operator $\mathbb{R}^n \rightarrow \mathbb{R}^n$.
The Green's function $G_0$ corresponding to the same problem
is given by the formula
\begin{equation}\label{eqn:greensfunc1}
G_0(x,t)=
\begin{cases}
I_n -(A+C)^{-1}\cdot C \;\;&(t\leq x)\;, \\
-(A+C)^{-1}\cdot C \;\; &(t\geq x) \;.
\end{cases}
\end{equation}
\end{lemma}

\begin{lemma}
For any $\lambda\in\mathbb{R}$, a fundamental matrix $\Gamma_0$ for the equation
\begin{equation} \label{eqn:background2}
\Phi'(x)+\lambda\cdot \Phi(x) = 0
\end{equation}
is given by the formula
\begin{equation} \label{eqn:fundmat2}
\Gamma_0(x) = e^{-\lambda x}\cdot I_n\;,
\end{equation}
where $n$ is the dimensionality of the problem \eqref{eqn:inhomo}, \eqref{bc:homo},
and $I_n$ denotes the unity operator $\mathbb{R}^n \rightarrow \mathbb{R}^n$.
The Green's function $G_0$ corresponding to the same problem
is given by the formula
\begin{equation}\label{eqn:greensfunc2}
G_0(x,t)=
\begin{cases}
e^{\lambda(t-x)}\cdot I_n -e^{\lambda(t-1)}(A+e^{-\lambda}\cdot C)^{-1}\cdot C \;\;&(t\leq x)\;, \\
-e^{\lambda(t-x-1)}\cdot(A+e^{-\lambda}\cdot C)^{-1}\cdot C \;\; &(t\geq x) \;.
\end{cases}
\end{equation}
\end{lemma}

We reiterate that in both cases the condition
\begin{equation} \label{eqn:nondegeneratebc}
det(D_0) \neq 0
\end{equation}
needs to be satisfied for the construction to be valid.
Here, $D_0$ is the boundary condition matrix defined in
\eqref{eqn:defd0}.
In the former case, it is equivalent to 
the condition
\begin{equation} \label{eqn:nondegeneratebc1}
det(A+C) \neq 0\;,
\end{equation}
while in the latter, it is reduced to
\begin{equation} \label{eqn:nondegeneratebc2}
det(A+e^{-\lambda}\cdot C) \neq 0\;.
\end{equation}

We make the further observation that neither \eqref{eqn:nondegeneratebc1}
nor \eqref{eqn:nondegeneratebc2} is a necessary condition for the existence or uniqueness of the solution to the boundary value problem. There exist boundary value problems with unique solutions where both conditions are violated.

Consider as an example a second-order scalar equation with Dirichlet boundary
condition:
\begin{equation} \label{eqn:2ndorderscalar}
\begin{split}
u''(x)+p(x)u'(x)&+q(x)u(x) = f(x) \\
u(a)=u_a\;&,\; u(c)=u_c\;.
\end{split}
\end{equation}

It is standard to introduce the change of variable 
\begin{equation} \label{eqn:changeofvar}
\Phi(x) = 
\left[
\begin{aligned}
u(x) \\
u'(x)
\end{aligned}
\right]
\end{equation}
to reduce it to a first-order system of dimension two.
After the transform, the boundary condition assumes
\begin{equation}\label{eqn:degeneratebc0}
A=
\begin{bmatrix}
1 & 0\\
0 & 0
\end{bmatrix}
\;\;
,
\;\;
C=
\begin{bmatrix}
0 & 0\\
1 & 0
\end{bmatrix}\;.
\end{equation}

It is obvious that in this case both conditions \eqref{eqn:nondegeneratebc1}
and \eqref{eqn:nondegeneratebc2} are violated (for all $\lambda\in\mathbb{R}$). This motivates us to consider
a slightly more general background equation.
\begin{lemma} \label{lem:fundmat3}
A fundamental matrix $\Gamma_0$ for the equation
\begin{equation} \label{eqn:background3}
\Phi'(x)+Q_0\cdot \Phi(x) = 0
\end{equation}
is given by the formula
\begin{equation} \label{eqn:fundmat3}
\Gamma_0(x) = e^{-Q_0(x-a)}\;,
\end{equation}
where $Q_0\in\mathbb{R}^{n\times n}$ is an arbitrary constraint matrix.
The Green’s function $G_0$ corresponding to the same problem is given by 
the formula
\begin{equation}\label{eqn:greensfunc3}
G_0(x,t)=
\begin{cases}
e^{-Q_0(x-a)}[e^{Q_0(t-a)}-(A+Ce^{-Q_0(c-a)})^{-1}Ce^{-Q_0(c-t)}] \;\;&(t\leq x)\;, \\
-e^{-Q_0(x-a)}(A+Ce^{-Q_0(c-a)})^{-1}Ce^{-Q_0(c-t)} \;\; &(t\geq x) \;.
\end{cases}
\end{equation}
\end{lemma}

\begin{remark}
Now the condition for the existence of a Green's function is reduced
to $det(A+Ce^{-Q_0(c-a)})\neq 0$, for a suitably chosen $Q_0$. Clearly, the boundary condition \ref{eqn:degeneratebc0} is allowed, taking $Q_0=\begin{bmatrix} 0 & 1 \\0 & 0 \end{bmatrix}$. In the general case, picking a random orthogonal matrix $U$ and letting $Q_0= U\Lambda U^*$, where $\Lambda = diag(\lambda_1, \lambda_2,\cdots, \lambda_n)$ works for most problems in practice. 
\end{remark}

\begin{remark}
A more subtle question is: How does the conditioning of the integral
equation change as the choice of matrix $Q_0$ varies? Since matrix inversion is involved in the formula \ref{eqn:greensfunc3}, an obvious necessary condition for the integral equation to be well-conditioned is that $D_0=A+Ce^{-Q_0(c-a)}$ must be well-conditioned.
But is it also a sufficient condition? We postpone the discussion
of this issue to section \ref{sec:casestudy}.
\end{remark}

In summary, we provided three different constructions of background Green's functions in this subsection. The first two, though simple and easy to implement, are insufficient for certain boundary conditions.
The last one, though generally applicable, introduces the
further complication of (i) the evaluation of matrix exponentials and (ii) potential change of conditioning. 

\subsection{A linear transform for degenerate boundary conditions}
In this subsection, let us deviate from the approach of
choosing different background problems and provide another
approach for the successful formulation of integral equations.

As the authors of \cite{starrrokhlin} observed, when
conditions \eqref{eqn:nondegeneratebc1} or \eqref{eqn:nondegeneratebc2} are violated, 
one can seek a linear transform $T(x): [a,c]\rightarrow \mathbf{L}(\mathbb{R}^{n\times n})$ that transforms the
equation and the boundary condition at the same time. 
We begin by citing the one theorem provided in \cite{starrrokhlin}:

\begin{theorem} \label{thm:tx}
Suppose $\Phi:[a,c]\rightarrow \mathbb{R}^n$ is the unique solution to the problem \eqref{eqn:inhomo}, \eqref{bc:homo}, and there exists some $T(x):[a,c]\rightarrow \mathbf{L}(\mathbb{R}^{n\times n})$ such that $T\in C^1([a,c])$ and $T(x)$ is nonsingular for all $x\in [a,c]$, then the (transformed) boundary value problem

\begin{align}
&\varphi'(x)+T^{-1}(x)\cdot(T'(x)+p(x)\cdot T(x))\cdot\varphi (x)=T^{-1}(x)f(x), \label{eqn:eqntrans}\\
&A\cdot T(a)\cdot\varphi(a)+C\cdot T(c)\cdot\varphi(c)=0. \label{eqn:bctrans}
\end{align} 

has a unique solution 
$$\varphi(x)=T^{-1}(x)\cdot\Phi(x) ,\quad x\in[a,c].$$
\end{theorem}


The conclusion of this theorem may seem obvious, and the proof
is no more than a linear algebra exercise. However, it does create
more freedom to deal with the so-called degenerate boundary conditions. We observe that as long as there exists a linear transform $T(x)$, which satisfies $det(A\cdot T(a) + C\cdot T(c))\neq 0$ in addition to the assumption of the above theorem, the simplest background Green function \eqref{eqn:greensfunc1} will work. The discussion of \cite{starrrokhlin} ends here and leaves the construction of such linear transforms to the user.

Fortunately, with a little more linear algebra, we figure out that
such a linear transform exists in general and can be constructed
explicitly. The main result is summarized in the following theorem.

\begin{theorem} [explicit construction of the linear transform]
\label{thm:constr_tx}
There exist $T:[a,c]\rightarrow \mathbb{R}^n$ such that $T\in C^1([a,c])$, $det(T(x))\neq 0,\ \forall x\in [a,c]$ and $det(A\cdot T(a)+C\cdot T(c))\neq 0$, if and only if 
\begin{equation} \label{eqn:gfun_cond1}
    N(A^T)\cap N(C^T)=\{0\}\;,
\end{equation}
or equivalently
\begin{equation}\label{eqn:gfun_cond2}
    R(A)+R(C)=\mathbb{R}^n\;.
\end{equation}
 
Moreover, $T(x)$ can be constructed as the product of (no more than $n$) orthogonal matrices and a diagonal matrix.
\end{theorem} 

\begin{proof}
For the necessity, notice that if $x\in N(A^T)\cap N(C^T)$, $x\cdot (A\cdot T(a)+C\cdot T(c))=0$, so $x=0$.
For the sufficiency, we construct $T$ assuming that $N(A^T)\cap N(C^T)=\{0\}$.

Assume $A=(a_1,a_2,\cdots,a_n)$, $C=(c_1,c_2,\cdots, c_n)$, 
$$N(A^T)\cap N(C^T)=\{0\}\Leftrightarrow span\{ a_1, a_2,\cdots,a_n,c_1, c_2, \cdots, c_n\}=\mathbb{R}^n.$$
So there exists $\{i_1,\cdots,i_r\},\{j_1,\cdots,j_{n-r}\}\subset \{1,2,\cdots,n\}$, such that $a_{i_1}, a_{i_2},\cdots, a_{i_r} ,\allowbreak c_{j_1}, c_{j_2}, \cdots, c_{j_{n-r}}$ are linearly independent.

Consider a permutation $\sigma\in S_n$, satisfying 
$$\{\sigma(j_1), \sigma(j_2), \cdots, \sigma(j_{n-r})\}\cup \{i_1,i_2,\cdots,i_r\} =\{1,2,\cdots,n\}.$$ 
We aim to find $T$ with $T(a)=I$, $T(c)=P_\sigma\cdot \Lambda$, where $P_\sigma$ is the permutation matrix corresponding to $\sigma$, and $\Lambda$ is the diagonal matrix, such that 
$$A\cdot T(a)+C\cdot T(c)=A+C\cdot P_\sigma\cdot \Lambda=
(a_1+\lambda_1\cdot c_{\sigma(1)},\cdots,a_n+\lambda_n\cdot c_{\sigma(n)})$$
There are infinitely many $\Lambda$ that make $A+C\cdot P_\sigma\cdot \Lambda$ nonsingular. One way to construct $\Lambda$ is to set
\begin{equation}
\left\{
\begin{aligned}
&|\lambda_k| =\lambda,\quad if \ k=\sigma(j_l),\ l=1,2,\cdots,n-r;\\
&|\lambda_k| =1/\lambda, \quad if \ k=i_l,\ l=1,2,\cdots,r.
\end{aligned}
\right.
\label{eqn:lambdai}
\end{equation}
Here the sign remains undetermined, which will be addressed in the next step. We note that as $\lambda$ approaches infinity, $A\cdot T(a)+C\cdot T(c)$ approaches a fixed matrix $(d_1,d_2,\cdots,d_n)$ after dividing by $\lambda$ in several rows, with $|d_{j_l}|=c_{j_l}$, $|d_{i_l}|=a_{i_l}$. The limiting matrix is nonsingular, with n independent columns, so regardless of the sign, there exists some $\lambda$ such that $det(A\cdot T(a)+C\cdot T(c))$ is nonsingular.

The next step is to find a continuous path in $GL_n(\mathbb{R})$, starting from the identity matrix $I$ while ending at the $T(c)=P_\sigma\cdot \Lambda$ we have constructed. First, we decompose the permutation matrix $P_\sigma$ into a composition of 2-permutation matrices:
$$P_\sigma=\prod^r_{i=1}P_i$$
where the matrix $P_i$ permutes the $s_i$ and $k_i$ row. $P_i$ and $I$ are not in the same path-connected component in $GL_n(\mathbb{R})$, but we can use the matrix 
\begin{equation}
    \widetilde{P}_i(x)=
\begin{bmatrix}
1      &        &        &        &  0      &        &        \\
       & \ddots &        &        &        &        &        \\
       &        & \cos\frac{\pi(x-a)}{2 (b-a)} &        & -\sin\frac{\pi(x-a)}{2 (b-a)} &        &        \\
       &        &        &       &        &        &        \\
       &        & \sin\frac{\pi(x-a)}{2 (b-a)} &        & \cos\frac{\pi(x-a)}{2 (b-a)} &        &        \\
       &        &        &        &        & \ddots &        \\
       &        &   0     &        &        &        & 1
\end{bmatrix}
\label{eqn: permutation}
\end{equation} 
to construct a path from $I$ to the 2-permutation with the sign of one element flipped, denoted by $\widetilde{P_i}$.

Let $\widetilde{P}_\sigma=\prod_{i=1}^r\widetilde{P_i}=P_\sigma\cdot \Lambda_{sign}$. $\Lambda_{sign}$ is a diagonal matrix with elements $\pm 1$ which determines signs of the matrix $\Lambda$. Assume $\Lambda=\Lambda_{sign}\cdot \widetilde{\Lambda}$, where $\widetilde{\Lambda}$ is a diagonal matrix with positive elements $(\widetilde{\lambda}_1,\widetilde{\lambda}_2,\cdots,\widetilde{\lambda}_n)$, $\widetilde{\lambda_i}=\lambda\ or\ 1/\lambda$. We can also construct a path from $I$ to $\widetilde{\Lambda}$ in $GL_n(\mathbb{R})$:
\begin{equation}
    \widetilde{\Lambda}(x)=
\begin{bmatrix}
\frac{\widetilde{\lambda}_1-1}{b-a}\cdot x+1 & 0         & \cdots    & 0 \\
0         &\frac{\widetilde{\lambda}_2-1}{b-a}\cdot x+1 & \cdots    & 0 \\
\vdots    & \vdots    & \ddots    & \vdots \\
0         & 0         & \cdots    &\frac{\widetilde{\lambda}_n-1}{b-a}\cdot x+1
\end{bmatrix}
\label{eqn: Lambda}
\end{equation}
Finally we have a path in $GL_n(\mathbb{R})$, $T(x)=\prod_{i=1}^r\widetilde{P_i}(x)\cdot \widetilde{\Lambda}(x)$. $T(a)=I$ and $T(c)=P_\sigma\cdot \Lambda_{sign}\cdot \widetilde{\Lambda}=P_\sigma\cdot \Lambda$, satisfying $A\cdot T(a)+C\cdot T(c)$ is a nonsingular matrix.
\end{proof}

\begin{remark}
Condition \eqref{eqn:gfun_cond2} is also a necessary condition for the existence of solution to the boundary value problem, for an arbitrarily given $\gamma\in\mathbb{R}^n$. When it is violated, an additional condition
$\gamma\in R(A)+R(C)$ is needed for the boundary condition to be consistent.
\end{remark}

\begin{remark}
    When constructing the diagonal matrix, we have taken the limit $\lambda\rightarrow \infty$
    in equation \eqref{eqn:lambdai}.
    In practice, this can be achieved by repeatedly doubling $\lambda$ until $A+C\cdot P_{\sigma}\cdot \Lambda$ is not singular. One concern is that $\lambda$ becomes too large, making $\Lambda$ ill-conditioned. There exist different methods to further transform the matrices to avoid this situation. However, in practice, we observe that this almost never happens since the boundary condition matrices $A$ and $C$ are far from arbitrary. As a result, we skip this part to simplify the algorithm.
\end{remark}

One nice feature of this theorem is that the proof is constructive
and directly leads to an algorithm. It can be viewed as a preprocess stage of the solver, that is, once the boundary condition is known, we construct such a linear transform $T(x)$ when needed, so that the transformed problem has a nondegenerate boundary condition. We summarize this algorithm by a pseudo-code (Algorithm \ref{alg:tx}).

\begin{algorithm} 
\caption{preprocess the boundary condition}
\label{alg:tx}
\begin{algorithmic}[1]  
\Require Matrices $A=(a_1,\dots,a_n)$ and $C=(c_1,\dots,c_n)$, satisfying Theorem~2.17
\Ensure $T:[a,c]\to \mathbb{R}^n$ such that $T\in C^1([a,c])$, 
       $\det T(x)\neq 0$ for all $x\in [a,c]$, and 
       $\det(A\cdot T(a)+C\cdot T(c))\neq 0$

\Statex \textbf{Step 1: Find the permutation $\sigma$}
\State Find index sets $\{i_1,\dots,i_r\},\{j_1,\dots,j_{n-r}\} \subset \{1,\dots,n\}$
\Statex such that $a_{i_1},\dots,a_{i_r},c_{j_1},\dots,c_{j_{n-r}}$ are linearly independent
\State Construct $\sigma$ so that 
       $\{\sigma(j_1),\dots,\sigma(j_{n-r})\} \cup \{i_1,\dots,i_r\} = \{1,\dots,n\}$
\State Decompose the permutation matrix $P_\sigma$ into 2-permutations: $P_\sigma = \prod_{i=1}^r P_i$

\Statex \textbf{Step 2: Find the index $\lambda_k$}
\State Set $\lambda \gets 1$, $\widetilde{P}_\sigma = \prod_{i=1}^r \widetilde{P}_i(1)$, where $\widetilde{P_i}$ are given by \ref{eqn: permutation}
\Repeat
    \State $\widetilde{\Lambda} \gets \operatorname{diag}(\lambda_1,\dots,\lambda_k)$, $\lambda_i$ are given by formula \ref{eqn:lambdai} and the sign is positive
    \State $D \gets A + C \cdot \widetilde{P}_\sigma \cdot \widetilde{\Lambda}$
    \State $\lambda\gets2\cdot \lambda$
\Until{$D$ is nonsingular}

\State \textbf{Step 3: Return the output}
\State \Return $T(x)=\prod_{i=1}^r\widetilde{P_i}(x)\cdot \Lambda(x)$, where $\Lambda(x)$ is given by \ref{eqn: Lambda}
\end{algorithmic}
\end{algorithm}

After constructing the linear transform, since the boundary condition is now nondegenerate, we can combine it with the simplest background Green's function \eqref{eqn:background1} to complete the integral equation formulation.

In fact, a little linear algebra reveals the relation between \eqref{eqn:gviagamma}
and \eqref{eqn:eqntrans}, \eqref{eqn:bctrans}.
Letting $T(x)=\Gamma_0(x)$ with a known fundamental solution $\Gamma_0(x)$, the background Green's function in \eqref{eqn:gviagamma} can be recovered by first transforming the equation according to \eqref{eqn:eqntrans}, \eqref{eqn:bctrans}, followed by evaluating the simplest background Green's function \eqref{eqn:background1} for the transformed boundary condition. As a result, different choices of the background Green function such as \eqref{eqn:greensfunc2} and \eqref{eqn:greensfunc3} can be viewed as different linear transforms in Theorems \ref{thm:tx} and \ref{thm:constr_tx}.
However, it is still interesting to consider these options, since (i) they provide a different viewpoint, and (ii) they are easy to implement and are applicable to a wide range of problems.

\section{Discretization and fast algorithms}
Using the methods discussed in the previous section, 
we have successfully converted the original boundary value problem
to an integral equation of the form \eqref{eqn:eqnsigma}. In this section, we briefly review the components needed for the numerical
solution of this equation. 

\subsection{Recursive decomposition of the integral operator} \label{sec:recur}
If we discretize the integral equation directly, the matrix of the resulting linear system is dense, requiring $O((nN)^3)$ work to solve directly. ($N$ denotes the total number of discretization nodes.) In \cite{greengardrokhlin} and \cite{starrrokhlin}, the authors proposed fast algorithms with $O(n^3N)$ complexity for equations of this type, making use of the hierarchical off-diagonal low rank structure of the matrix. The last few decades have seen
substantial development of fast algorithms. The algorithms of \cite{greengardrokhlin} and \cite{starrrokhlin} now fall into the category of fast direct solvers (see ref.), but tend to achieve better performance than a
general purpose fast direct solver. For the sake of completeness, we give a concise review of the main results, focusing on exploiting the special structure of the integral operator. Interested readers are recommended to
\cite{greengardrokhlin} and \cite{starrrokhlin} for details. 

We begin by introducing some notations: let $P: (L^2[a,b])^n \rightarrow (L^2[a,b])^n$ be the integral operator
\begin{equation} \label{eqn:defp}
    P[\sigma](x) = \sigma(x)+\tilde{p}(x)\cdot\int_a^cG_0(x,t)\cdot\sigma(t)dt\;,
\end{equation}
where $\tilde{p}(x)=p(x)-p_0(x)$.

Then equation \eqref{eqn:eqnsigma} assumes a simple form
\begin{equation} \label{eqn:operatoreqn}
    P[\sigma](x) = f(x)\;.
\end{equation}

If $b\in [a,c]$, we partition the interval into two subintervals:
$[a,c]=A\cup B$, where $A=[a,b]$, and $B=[b,c]$. Here $A$ denotes a subinterval. (This notation is used in this section only and is not to be confused with the boundary condition matrix.)
Once a partition of the interval is given, the operator $P$ can be decomposed accordingly into four parts
$P_{AA} : (L^2[a,b])^n \rightarrow (L^2[a,b])^n $, 
$P_{AB} : (L^2[b,c])^n \rightarrow (L^2[a,b])^n $, 
$P_{BA} : (L^2[a,b])^n \rightarrow (L^2[b,c])^n $,
$P_{BB} : (L^2[b,c])^n \rightarrow (L^2[b,c])^n $,
which are defined by:
\begin{equation} \label{eqn:block}
    \begin{aligned}
        &P_{AA}(\sigma)(x)=\sigma(x)+\tilde{p}(x)\cdot\int_a^bG_0(x,t)\cdot\sigma(t)dt,\,\,(x\in A)\\
        &P_{AB}(\sigma)(x)=\Psi(x)\cdot\int_b^c\nu_R(t)\cdot \sigma(t)dt, \,\, (x\in A) \\
        &P_{BA}(\sigma)(x)=\Psi(x)\cdot\int_a^b\nu_L(t)\cdot \sigma(t)dt,\,\, (x\in B) \\
        &P_{BB}(\sigma)(x)=\sigma(x)+p(x)\cdot\int_b^cG_0(x,t)\cdot\sigma(t)dt,\,\, (x\in B)\;,    
    \end{aligned}
\end{equation}
where $\Psi(x)=\tilde{p}(x)\cdot \Gamma_0(x)$. For $P_{AB}$ and $P_{BA}$,
we have substituted the piecewise separation of variable form \eqref{eqn:gviagamma} of Green's function $G_0(x,t)$ into the operator
to reveal the fact that they are rank-$n$ operators.
The operators $P_{AA}$ and $P_{BB}$, on the other hand, assume
the same form as the global operator $P$, inheriting its properties.
We further introduce the following notation: $\sigma_A := \sigma_{|A}$, $\sigma_B := \sigma_{|B}$,
$f_A := f_{|A}$, $f_B := f_{|B}$.
Then \eqref{eqn:operatoreqn} is reduced to
\begin{equation} \label{eqn:block_op_simp}
    \begin{aligned}
        P_{AA} [\sigma_A] (x) + P_{AB} [\sigma_B] (x) &= f_A \\
        P_{BA} [\sigma_A] (x) + P_{BB} [\sigma_B] (x) &= f_B\;. 
    \end{aligned}
\end{equation}

\vspace{1cm}
\begin{minipage}{0.4\textwidth}
  \centering
  \begin{tikzpicture}[scale=1.3]
      \fill[blue!20] (0,2) rectangle ++(2,2);
      \fill[blue!20] (2,0) rectangle ++(2,2);
      
      \fill[green!20] (0,0) rectangle ++(2,2);
      \fill[green!20] (2,2) rectangle ++(2,2);
      
      \draw[step=2cm, thin] (0,0) grid (4,4);
    
      \node at (1,3) {$P_{aa}$};
      \node at (3,3) {$P_{ab}$};
      \node at (3,1) {$P_{bb}$};
      \node at (1,1) {$P_{ba}$};
  \end{tikzpicture}
  
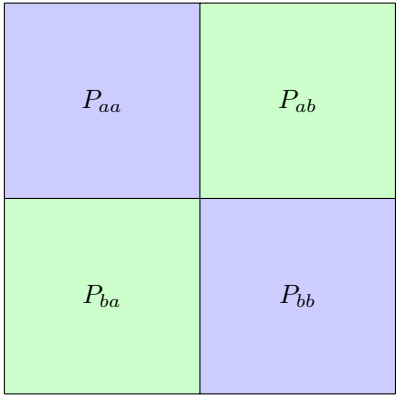
\captionof{figure}{Decomposition of the operator P}
  \label{partition}
\end{minipage}
\hfill
\begin{minipage}{0.4\textwidth}
  \centering
  \begin{tikzpicture}[scale=1.3]
  \fill[green!20] (0,0) rectangle ++(4,4);

  \foreach \x in {(0,3.5),(0.5,3),(1,2.5),(1.5,2),(2,1.5),(2.5,1),(3,0.5),(3.5,0)} {
      \fill[blue!20] \x rectangle ++(0.5,0.5);
  }

  \draw[step=2cm, thin] (0,0) grid (4,4);
  \draw[step=1cm, thin] (2,0) grid (4,2);
  \draw[step=1cm, thin] (0,2) grid (2,4);
  \foreach \x in {(0,3),(1,2),(2,1),(3,0)}{
  \draw[step=0.5cm, thin] \x grid ++(1,1);
  }

  \end{tikzpicture}
  
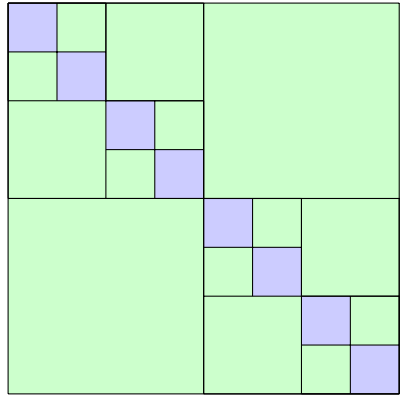
\captionof{figure}{hierarchical structure of P}
  \label{structure}
\end{minipage}

Block elimination can be applied to further simplify equation \eqref{eqn:block_op_simp}. To state the main result, we need to introduce some auxiliary variables as follows.
 \begin{equation} \label{def:etaa}
     \eta_A: [a,b]\rightarrow \mathbb{R}^n:\quad \eta_A=P_{AA}^{-1}f_A, 
 \end{equation}
 \begin{equation} \label{def:etab}
     \eta_B: [b,c]\rightarrow \mathbb{R}^n:\quad \eta_B=P_{BB}^{-1}f_B,
 \end{equation}
 \begin{equation} \label{def:phi}
     \varphi: [a,c]\rightarrow \mathbb{R}^{n\times n}:\quad \varphi = P^{-1}(\Psi),
 \end{equation}
 \begin{equation} \label{def:phia}
     \phi_A : [a,b]\rightarrow \mathbb{R}^{n\times n}:\quad \phi_A = P^{-1}_{AA}(\Psi|_A)
 \end{equation}
  \begin{equation} \label{def:phib}
     \phi_B : [b,c]\rightarrow \mathbb{R}^{n\times n}:\quad \phi_B = P^{-1}_{BB}(\Psi|_B)
 \end{equation}
 \begin{equation} \label{def:ga}
     g_A: (L^2[a,b])^n\rightarrow \mathbb{R}^n:\quad g_A(\eta)=\int_a^b\nu_L(t)\cdot \eta(t)dt
 \end{equation}
 \begin{equation} \label{def:gb}
     g_B: (L^2[b,c])^n\rightarrow \mathbb{R}^n:\quad g_B(\eta)=\int_b^c\nu_R(t)\cdot \eta(t)dt
 \end{equation}
For \eqref{def:phi} - \eqref{def:phib}, we have extended the operator
$P$ from $(L^2[a,c])^n\rightarrow (L^2[a,c])^n$ to that of $(L^2[a,b])^{n\times n}\rightarrow (L^2[a,b])^{n\times n}$, by replacing the matrix vector product by a matrix matrix
product. (The same extension is applied to $P_{AA}$ and $P_{BB}$.) To simplify the notation, we still denote it by $P$. Now we are ready to 
state the main result that enables the fast algorithm.

\begin{lemma} \label{lem:block1}
If $P$, $P_{AA}$ and $P_{BB}$ are all invertible, using the above notation,
the solution $\sigma=P^{-1}f$ is reduced to:
\begin{equation} \label{lem:sig_pcws}
    \begin{aligned}
        &\sigma|_A=\eta_A-\phi_A\cdot(I_n-g_A(\phi_A)\cdot g_B(\phi_B))^{-1}\cdot( g_B(\phi_B)\cdot g_A(\eta_A)-g_B(\eta_B))\;,\\
        &\sigma|_B=\eta_B-\phi_B\cdot(I_n-g_B(\phi_B)\cdot g_A(\phi_A))^{-1}\cdot( g_A(\phi_A)\cdot g_B(\eta_B)-g_A(\eta_A))\;.
    \end{aligned}
\end{equation}
\end{lemma}

A similar conclusion holds for $\varphi=P^{-1}(\Psi)$.

\begin{lemma} \label{lem:block2}
If $P$, $P_{AA}$ and $P_{BB}$ are all invertible, using the above notation, the solution $\varphi=P^{-1}(\Psi)$ is reduced to:
\begin{equation}
    \begin{aligned}
        &\varphi|_A=\phi_A\cdot(I_n-g_A(\phi_A)\cdot g_B(\phi_B))^{-1}\cdot(I_n-g_B(\phi_B))\;,\\
        &\varphi|_B=\phi_B\cdot(I_n-g_B(\phi_B)\cdot g_A(\phi_A))^{-1}\cdot(I_n-g_A(\phi_A))\;.
        \label{lemma3}
    \end{aligned}
\end{equation}
\end{lemma}

Interested readers are recommended to \cite{starrrokhlin} for proofs of Lemma \ref{lem:block1} and Lemma \ref{lem:block2}. Lemma \ref{lem:block1} reduces
the solution $\sigma = P^{-1}f$ to a linear combination of $\eta_A$ and $\phi_A$ on the subinterval $A$, and that of $\eta_B$ and $\phi_B$ on $B$.
Lemma \ref{lem:block2} reduces the solution $\varphi=P^{-1}\Psi$ to a piecewise representation of $\phi_A$ and $\phi_B$. Recursive use of the two lemmas leads to a fast algorithm on the continuous level. 

Assume that we have a recursive bisection of the interval $[a,c]$, which
corresponds to a binary tree structure. Let $B_i=[a_i,c_i]$ be an interval
on the finest level, so that $[a,c]=\cup_i^M [a_i,c_i]$. We begin by solving the local equations $P_{B_i,B_i} \eta_{B_i} = f_{B_i}$ and $P_{B_i,B_i} \phi_{B_i} = \Psi|_{B_i}$,
for $i=1, 2, \cdots, M$. Lemma \ref{lem:block1} and lemma \ref{lem:block2} enable us to go up the tree, forming solutions to the two equations on each node by patching the solutions on its children nodes. When the root is reached,
$\sigma = \eta_{[a,c]}$ gives the solution to the global integral equation \eqref{eqn:operatoreqn}. An illustration of this procedure is given in
\ref{structure}.

In this fashion, we can successfully reduce the task of solving global equations to that of solving local equations, which follows the key idea of fast direct solvers. Here we have not introduced any numerical approximation yet. Everything is exact and works on the continuous level. A closer examination of the constant matrices in the above lemmas reveals similar recursive relations, further eliminating the cost of forming solutions for parent nodes, accelerating the algorithm. To state the results, we need to introduce more auxiliary variables.

We define six constant matrices $\alpha_L^A, \alpha_R^A, \alpha_L^B, \alpha_R^B, \alpha_L, \alpha_R$ as follows:
 \begin{equation} \label{def:ala}
     \alpha_L^A \, = \, \int_a^b \nu_L(t)\cdot \phi_A(t)\;dt \;, 
 \end{equation}
 \begin{equation} \label{def:ara}
     \alpha_R^A \, = \, \int_a^b \nu_R(t)\cdot \phi_A(t)\;dt \;, 
 \end{equation}
 \begin{equation} \label{def:alb}
     \alpha_L^B \, = \, \int_b^c \nu_L(t)\cdot \phi_B(t)\;dt \;, 
 \end{equation}
 \begin{equation} \label{def:arb}
     \alpha_R^B \, = \, \int_b^c \nu_R(t)\cdot \phi_B(t)\;dt \;, 
 \end{equation}
  \begin{equation} \label{def:al}
     \alpha_L \, = \, \int_a^c \nu_L(t)\cdot \varphi(t)\;dt \;, 
 \end{equation}
 \begin{equation} \label{def:ar}
     \alpha_R \, = \, \int_a^c \nu_R(t)\cdot \varphi(t)\;dt \;, 
 \end{equation}

 and six constant vectors $\delta_L^A, \delta_R^A, \delta_L^B, \delta_R^B, \delta_L, \delta_R$ as follows:

 \begin{equation} \label{def:dla}
     \delta_L^A \, = \, \int_a^b \nu_L(t)\cdot \eta_A(t)\;dt \;, 
 \end{equation}
 \begin{equation} \label{def:dra}
     \delta_R^A \, = \, \int_a^b \nu_R(t)\cdot \eta_A(t)\;dt \;,
 \end{equation}
 \begin{equation} \label{def:dlb}
     \delta_L^B \, = \, \int_b^c \nu_L(t)\cdot \eta_B(t)\;dt \;,
 \end{equation}
 \begin{equation} \label{def:drb}
     \delta_R^B \, = \, \int_b^c \nu_R(t)\cdot \eta_B(t)\;dt \;,
 \end{equation}
  \begin{equation} \label{def:dl}
     \delta_L \, = \, \int_a^c \nu_L(t)\cdot \sigma(t)\;dt \;, 
 \end{equation}
 \begin{equation} \label{def:dr}
     \delta_R \, = \, \int_a^c \nu_R(t)\cdot \sigma(t)\;dt \;.
 \end{equation}

Now we are ready to state the main results that enable the acceleration
of the algorithm.

\begin{theorem}
Assume that the operators $P$, $P_{AA}$, $P_{BB}$ (as defined in this section) are all nonsingular. Assume further that a function $F(x)$ is defined by
the formula
\begin{equation}
    F(x) = \sigma(x) + \varphi(x)\cdot \lambda\;,
\end{equation}
for a constant vector $\lambda\in\mathbb{R}^n$. Then on the subintervals $A$ and $B$, we have:

\begin{align}
    F|_A(x) & = \eta_A(x) + \phi_A(x)\cdot \lambda_L\;, \label{eqn:fa} \\
    F|_B(x) & = \eta_B(x) + \phi_B(x)\cdot \lambda_R\;,  \label{eqn:fb}
\end{align}
where 
\begin{align}
    \lambda_L & = \Delta_2^{-1}(\lambda-\alpha_R^B\cdot (\lambda-\delta_L^A)-\delta_R^B)\;, \label{eqn:lambdal} \\
    \lambda_R & = \Delta_1^{-1}(\lambda-\alpha_L^A\cdot (\lambda-\delta_R^B)-\delta_L^A)\;,  \label{eqn:lambdar}
\end{align}
and 
\begin{align}
    \Delta_1 & = I_n -\alpha_L^A\cdot \alpha_R^B\;, \label{eqn:delta1} \\
    \Delta_2 & = I_n -\alpha_R^B\cdot \alpha_L^A\;.  \label{eqn:delta2}
\end{align}
\end{theorem}

The coefficients $\lambda, \lambda_L, \lambda_R$ are referred to as coupling
coefficients. A recursive relation of the coupling coefficients is given
in terms of the constant matrices and vectors defined above. 

\begin{theorem} \label{thm:recur_alpha}
Assume that the operators $P$, $P_{AA}$, $P_{BB}$ (as defined in this section) are all nonsingular, then
\begin{align}
    \alpha_L &= \alpha_L^A\cdot \Delta_2^{-1}\cdot (I_n-\alpha_R^B)
    + \alpha_L^B\cdot \Delta_1^{-1}\cdot (I_n-\alpha_L^A) \;, \label{eqn:recur_alphal} \\
    \alpha_R &= \alpha_R^A\cdot \Delta_2^{-1}\cdot (I_n-\alpha_R^B)
    + \alpha_R^B\cdot \Delta_1^{-1}\cdot (I_n-\alpha_L^A)  \;. \label{eqn:recur_alphar}
\end{align}
\end{theorem}

\begin{theorem} \label{thm:recur_delta}
Assume that the operators $P$, $P_{AA}$, $P_{BB}$ (as defined in this section) are all nonsingular, then
\begin{align}
    \delta_L & = \delta_L^A + \delta_L^B +\alpha_L^A\cdot \Delta_2^{-1}
    \cdot (\alpha_R^B\cdot \delta_L^A - \delta_R^B) + \alpha_L^B\cdot
    \Delta_1^{-1}\cdot (\alpha_L^A\cdot \delta_R^B-\delta_L^A)\;, \label{eqn:recur_deltal} \\
    \delta_R &=  \delta_R^A + \delta_R^B +\alpha_R^A\cdot \Delta_2^{-1}
    \cdot (\alpha_R^B\cdot \delta_L^A - \delta_R^B) + \alpha_R^B\cdot
    \Delta_1^{-1}\cdot (\alpha_L^A\cdot \delta_R^B-\delta_L^A)\;. \label{eqn:recur_deltar}
\end{align}
\end{theorem}

We skip the proofs and refer the readers to \cite{starrrokhlin} for more details.

\subsection{High-order Nystr\"om discretization}
The last component for the design of a fast algorithm is a proper discretization of the integral equation.
The analytic apparatus developed in the previous subsection reduces the task to the solution of local integral equations on each subinterval $[a_i,c_i]$:
\begin{equation}\label{eqn:subinteg1}
    P_i(\eta_i)(x)=\eta_i(x)+\tilde{p}(x)\cdot\int_{a_i}^{c_i}G_0(x,t)\cdot\eta_i(t)dt=f|_{[a_i,c_i]}(x)\;. 
\end{equation}
and 
\begin{equation}\label{eqn:subinteg2}
    P_i(\phi_i)(x)=\phi_i(x)+p(x)\cdot\int_{a_i}^{c_i}G_0(x,t)\cdot\phi_i(t)dt=\Psi|_{[a_i,c_i]}(x).
\end{equation}

Here we discretize equations \eqref{eqn:subinteg1} and \eqref{eqn:subinteg2} via a Nystr\"om discretization based on $p$-point Chebyshev approximation and integration (also known as
Clenshaw-Curtis quadrature). 
It is a classical result that the order of convergence of a Nystr\"om method equals
the order of the underlying quadrature rule.
The numerical apparatus used here are all classical and can be found in texts such as [refs]. For the sake of completeness, we list the very basics of the numerical tools.

On an interval $[a_i,c_i]$, the (scaled) Chebyshev nodes are given by:
\begin{equation}
    t_i=(\frac{c_i-a_i}{2})cos(\frac{(2i-1)\pi}{2p})+\frac{a_i+c_i}{2}\;,
\end{equation}
which are the roots of the $p$-th order (scaled) Chebyshev polynomial $T_p(x)$ defined by
\begin{equation}
    T_p(x)=\cos (p\cdot\arccos(\frac{2(x-a_i)}{c_i-a_i}-1))\;,
\end{equation}

If the function values $\{f(t_j),\;j=1, \cdots, p\}$ are given for a function $f$ on the interval $[a_i,c_i]$,
its Chebyshev interpolant is given by:
\begin{equation}
 \bar{f}(x)\approx \sum_{i=0}^{p-1}\alpha_i T_i(x)\;,
\end{equation}
where the coefficients are
\begin{equation}
    \begin{aligned}
        \alpha_0 &=\frac{1}{p}\sum_{j=1}^pf(t_j);\\
        \alpha_i &=\frac{2}{p}\sum_{j=1}^pf(t_j)T_i(t_j)\quad (for\ i\geq 1)\;.
        \nonumber
    \end{aligned}
\end{equation}
Here $(\alpha_0,\alpha_1,\cdots,\alpha_{p-1})$ is called the Chebyshev coefficients of $f$. To fully discretize \eqref{eqn:subinteg1} and \eqref{eqn:subinteg2}, we need to also expand integrals of the form $F_l(x)=\int_{-1}^xf(t)dt$ or $F_r(x)=\int_x^1f(t)dt$ into Chebyshev series.
This can be achieved by the following lemma.

\begin{lemma}
    Suppose that $f\in C[-1,1]$ has a Chebyshev series
    \begin{equation}
        f(x)=\sum_{k=0}^{p-1} \alpha_k T_k(x)\;.
    \end{equation}
    Then the Chebyshev coefficients of
    $F_l(x)=\int_{-1}^x f(t)dt = \sum_{k=0}^{p-1} \alpha_k^l T_k(x)$ are given by
    \begin{equation}
    \begin{aligned}
        &\alpha^l_k=\frac{1}{2k}(\alpha_{k-1}-\alpha_{k+1})\quad for \ k=2,\cdots,p-1;\\
        &\alpha^l_1=\frac{1}{2}(2\alpha_0-\alpha_2);\\
        &\alpha_0^l=\sum_{k=1}^{p-1} (-1)^{k-1}\alpha_k.
        \label{cheb1}
    \end{aligned}
\end{equation}
\end{lemma}
Here we define $\alpha_p = 0$. Similarly, the Chebyshev coefficients of $F_r=\int_x^1 f(t)dt= \sum_{k=0}^{p-1} \alpha_k^r T_k(x)$ are given by
\begin{equation}
    \begin{aligned}
        &\alpha^r_k=\frac{1}{2k}(\alpha_{k+1}-\alpha_{k-1})\quad for \ k=2,\cdots,p-1;\\
        &\alpha^r_1=\frac{1}{2}(-2\alpha_0+\alpha_2);\\
        &\alpha_0^r=-\sum_{k=1}^p \alpha_k. 
        \label{cheb2}
    \end{aligned}
\end{equation}

Proof of this theorem can be found in classical texts such as [refs].

Definite integrals can easily be done by integrating the Chebyshev polynomials:
\begin{equation}
    \int_{-1}^1 \sum_{i=0}^{p-1}\alpha_i T_i(x)\, dx = 2\alpha_0 + \sum_{\substack{k=2 \\ k\text{ even}}}^{p-1} \alpha_k \cdot \frac{2(-1)^{k/2}}{1 - k^2}\;.
    \label{chebint}
\end{equation}

Replacing the integrals in \eqref{eqn:subinteg1} and \eqref{eqn:subinteg2} by the quadrature scheme presented above, and enforcing the equations at the same set of nodes leads to linear systems 
\begin{equation}\label{eqn:linsys1}
    \bar{P_i}\bar{\eta_i}=\bar{f}_i\;,
\end{equation}
and 
\begin{equation}\label{eqn:linsys2}
    \bar{P_i}\bar{\phi_i}=\bar{\Psi}_i\;.
\end{equation}

Local linear systems \eqref{eqn:linsys1} and \eqref{eqn:linsys2} are solved directly by Gaussian elimination and the solution to the global integral equation \eqref{eqn:operatoreqn} is recovered by recalling the recursive relations provided in subsection \ref{sec:recur}. This completes the full algorithm. More details of implementation can be found in [refs].

\section{A case study of conditioning} \label{sec:casestudy}
After describing the full algorithm, we now turn to a more subtle problem: conditioning of different integral formulations. 
For this we carry out a numerical study with three boundary value problems.

\textbf{Problem 1. Bessel's equation \cite{handbook}}
\begin{equation}
    \left\{
		\begin{aligned}
		  &u''+\frac{u'}{x}+\frac{x^2-100^2}{x^2}u=0,\\
	   	&u(0)=0,\ u(600)=1\;.
		\end{aligned}
    \right.
    \label{eqn:scalarBessel}
\end{equation}
Introducing the standard change of variable,
\begin{equation} \label{eqn:BesselPhi}
\Phi(x) = 
\left[
\begin{aligned}
u(x) \\
u'(x)
\end{aligned}
\right],
\end{equation}
we convert the equation\eqref{eqn:scalarBessel} into the system's form:
\begin{equation}\label{eqn:systemBessel}
    \left\{
    \begin{aligned}
        &\Phi'+\begin{bmatrix}
            0\ &-1\\
            \frac{x^2-100^2}{x^2}&\frac{1}{x}
        \end{bmatrix}
        \Phi = 
        \begin{bmatrix}
            0\\
            0
        \end{bmatrix},\\
        &\begin{bmatrix}
            1 &0\\
            0 &0
        \end{bmatrix}
        \Phi(0) + 
        \begin{bmatrix}
            0 &0\\
            1 &0
        \end{bmatrix}
        \Phi(600) =
        \begin{bmatrix}
            0\\
            1
        \end{bmatrix}\;.
    \end{aligned}
    \right.
\end{equation}

\textbf{Problem 2.}
\begin{equation}\label{eqn:u=sinx}
    \left\{\begin{aligned}
    &\Phi'+\begin{bmatrix}
        0&-1\\
        1&0
    \end{bmatrix}
    \Phi =\begin{bmatrix}
        0\\
        0
    \end{bmatrix},\\
    &\begin{bmatrix}
        1&0\\
        0&0
    \end{bmatrix}
    \Phi(0) + \begin{bmatrix}
        0&0\\
        1&0
    \end{bmatrix}
    \Phi(600) =\begin{bmatrix}
        \sin(0)\\
        \sin(600)
    \end{bmatrix}\;.
    \end{aligned}
    \right.
\end{equation}

\textbf{Problem 3.}
\begin{equation}\label{eqn:u=sinx/600}
    \left\{\begin{aligned}
    &\Phi'+\begin{bmatrix}
        0&-\frac{1}{600}\\
        \frac{1}{600}&0
    \end{bmatrix}
    \Phi =\begin{bmatrix}
        0\\
        0
    \end{bmatrix},\\
    &\begin{bmatrix}
        1&0\\
        0&0
    \end{bmatrix}
    \Phi(0) + \begin{bmatrix}
        0&0\\
        1&0
    \end{bmatrix}
    \Phi(600) =\begin{bmatrix}
        \sin(0)\\
        \sin(1)
    \end{bmatrix}.
    \end{aligned}
    \right.
\end{equation}

All three equations are defined on $[0,600]$.
Problem 1 is the classical Bessel's equation. Equation \eqref{eqn:scalarBessel} is the standard form of a second-order scalar boundary value problem. We convert it to the equivalent form \eqref{eqn:systemBessel} before applying our algorithm to it. This equation is degenerate at $x=0$ (when a factor of $x^2$ is multiplied to both sides of the equation, the coefficient of the leading order term vanishes at $x=0$) and has an analytic solution(see Fig. \ref{fig:bessel})

\begin{equation}\label{eqn:solbvp1}
\Phi(x) = 
\frac{1}{J_{100}(600)}\left[
\begin{aligned}
J_{100}(x) \\
J_{100}'(x)
\end{aligned}
\right]\;.
\end{equation}

It can be observed that the solution contains many oscillations on the interval $[0,600]$. 

Problem 2 and problem 3 are constructed manually, with exact solutions 
\begin{equation}\label{eqn:solbvp2}
    \Phi(x) = 
\left[
\begin{aligned}
\sin(x) \\
\cos(x)
\end{aligned}
\right]\;,
\end{equation}
and 
\begin{equation}\label{eqn:solbvp3}
    \Phi(x) = 
\left[
\begin{aligned}
\sin\left(\frac{x}{600}\right) \\
\cos\left(\frac{x}{600}\right)
\end{aligned}
\right]\;,
\end{equation}
respectively. We observe that they are simpler than Problem 1 in the sense that they are nondegenerate on the interval $[0,600]$. But the solution of Problem 2 is as oscillatory as that of Problem 1, while the solution of Problem 3 has only one oscillation. Together, they form a testing suite with hard and easy problems.
To study the conditioning, we solve the above problems with four different integral formulations.

\textbf{Formulation 1.}
The first formulation follows Lemma \ref{lem:fundmat3}, choosing a constant coefficient equation as the background problem.

Specifically, we choose $Q_0 = R \Sigma R $, where $R$ is an orthogonal matrix generated by the QR decomposition of a random matrix, and $\Sigma$ is a diagonal matrix given by:
\begin{equation} \label{eqn:Sigma}
    \Sigma=\begin{bmatrix}
        \frac{1}{600}&0\\
        0&\frac{2}{600}
    \end{bmatrix}.
\end{equation} 
In this case, the fundamental matrix is simply $\Gamma_0 = R \cdot \exp(-x \Sigma) \cdot R$. The background Green's function can be evaluated by formula \eqref{eqn:greensfunc3}.

\textbf{Formulation 2.}
In this case, we follow theorem \ref{thm:tx} and construct the linear transform $T_2(x)$ manually as:

\begin{equation}
    T_2(x)=\begin{bmatrix}
        1&x\\
        0&1
    \end{bmatrix}\;.
\end{equation}
We then use the simplest Green's function \eqref{eqn:greensfunc1} to convert the problem to an integral equation. 

\textbf{Formulation 3.}
This formulation is the same as Formulation 2, except that in this case, the linear transform $T_3(x)$ is constructed as:
\begin{equation}
        T_3(x)=\begin{bmatrix}
        1&\frac{x}{600}\\
        0&1
    \end{bmatrix}\;.
\end{equation}

\textbf{Formulation 4.}
As the last formulation, we call Algorithm \ref{alg:tx} directly to construct a linear transform $T_4(x)$ in a black-box fashion.
The simplest Green's function \eqref{eqn:greensfunc1} is adopted to convert the problem to an integral equation.

\begin{figure}[t]
\centering
\includegraphics[width=0.8\textwidth]{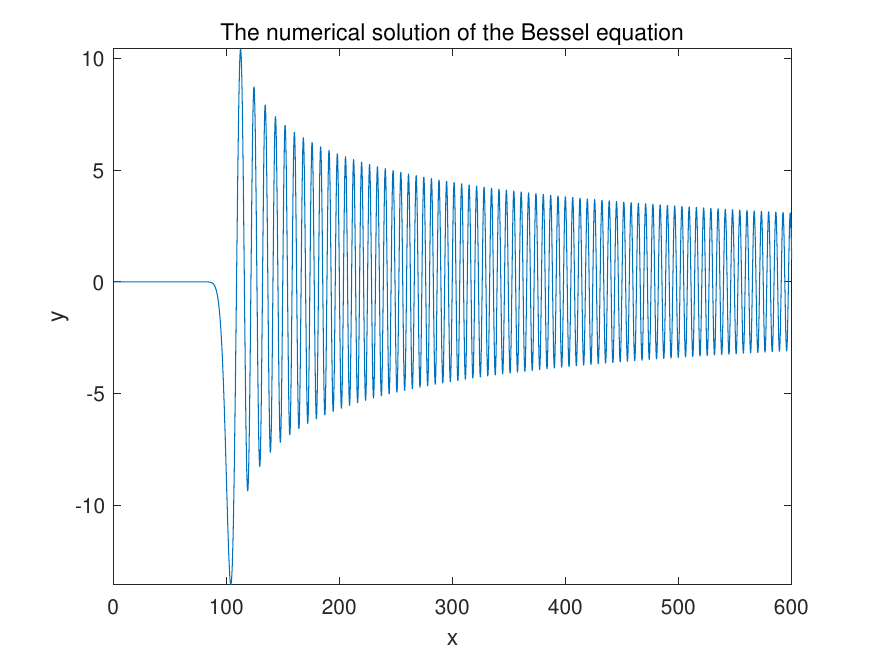}
\caption{The numerical result of Bessel equation.}
\label{fig:bessel}
\end{figure}

\begin{table}
    \centering
    \begin{tabular}{cccc}
        \hline
         & $\mathcal{L}^2$ relative error &  $\kappa_2(\bar{P})$ & $\max_{x\in[a,c]}\kappa_2(T(x))$ \\
        \hline
        $1$ & $1.50\times10^{-13}$ & $3.03\times10^{9}$ & $2.80\times 10^0$\\ 
        $2$ & $1.04\times10^{-7}$ & $7.91\times10^{17}$ & $3.60\times10^5$ \\    
        $3$ & $2.82\times10^{-12}$ & $1.29\times10^{11}$ & $2.62\times 10^0$\\ 
        $4$ & $2.65\times 10^{-12}$ &  $3.01\times 10^9$  & $1.00\times 10^0 $\\
        \hline
    \end{tabular}
    \caption{Numerical performance of three formulations solving the Bessel equation in system form, the first and second data are solved using the transformation $T_1(x)$ and $T_2(x)$ respectively and the third group using the new Green's function defined by \eqref{eqn:Sigma} to solve the problem.}
    \label{tab:besselnum}
\end{table}

\begin{table}
    \centering
    \begin{tabular}{cccc}
        \hline
         & $\mathcal{L}^2$ relative error &  $\kappa_2(\bar{P})$ & $\max_{x\in[a,c]}\kappa_2(T(x))$ \\
        \hline
        $1$ & $9.35\times10^{-12}$ & $8.42\times10^{7}$ & $2.89\times 10^0$\\  
        $2$ & $1.89\times10^{-7}$ & $1.13\times10^{17}$ & $3.60\times10^5$ \\    
        $3$ & $6.27\times10^{-12}$ & $3.38\times10^{7}$ & $2.62\times 10^0$\\   
        $4$ & $3.55\times 10^{-11}$ & $1.56\times 10^7$ & $1.00\times 10^0 $\\
        \hline
    \end{tabular}
    \caption{Numerical performance of three formulations solving problem \eqref{eqn:u=sinx}, the first and second data are solved using the transformation $T_1(x)$ and $T_2(x)$ respectively and the third group using the new Green's function defined by \eqref{eqn:Sigma} to solve the problem.}
    \label{tab:sinx}
\end{table}

\begin{table}
    \centering
    \begin{tabular}{cccc}
        \hline
         & $\mathcal{L}^2$ relative error & $\kappa_2(\bar{P})$ & $\max_{x\in[a,c]}\kappa_2(T(x))$ \\
        \hline
        $1$ & $1.23\times10^{-15}$ & $2.35\times10^{1}$ & $2.92\times 10^0$\\  
        $2$ & $3.29\times10^{-12}$ & $4.52\times10^{10}$ & $3.60\times10^5$ \\ 
        $3$ & $5.84\times10^{-16}$ & $2.96\times 10^0$ & $2.62\times 10^0$\\ 
        $4$ & $1.89\times 10^{-16}$ & $2.90\times 10^{0}$ & $1.00\times 10^0$ \\
        \hline
    \end{tabular}
    \caption{Numerical performance of three formulations solving problem \eqref{eqn:u=sinx/600}, the first and second data are solved using the transformation $T_1(x)$ and $T_2(x)$ respectively and the third group using the new Green's function defined by \eqref{eqn:Sigma} to solve the problem.}
    \label{tab:sinx/600}
\end{table}

The numerical results for each problem are given in Table \ref{tab:besselnum}
- Table \ref{tab:sinx/600}, where each row corresponds to a formulation listed above. For Problem 1 and Problem 2, we decompose the interval $[a,c]$ uniformly into $200$ subintervals $[a,c]=\cup _{i=1}^{200}[a_i,c_i]$, where $16$ Chebyshev nodes are imposed on each subinterval. For Problem 3, we use $50$ uniform subintervals instead. For the experiments in this section, we have not used any fast algorithm. The integral equation is discretized and solved directly by Gaussian elimination, so as to eliminate the effect of the fast algorithm.
Self-convergence study shows that all the quantities studied have converged. 

In the first column, we list the $L^2$ relative error in the solution. In the second column, we list the condition number (also in $L^2$ norm) of the system's matrix $\bar{P}$. In the last column, we list the maximum condition number of the transform $T(x)$ or the fundamental solution $\Gamma_0(x)$. 

As the tables suggest, the total number of lost digits (for each problem and each formulation) is proportional to the maximum condition number of $T(x)$ or $\Gamma_0(x)$, although $\kappa_2(\bar{P})$ increases to an even larger number. Different numbers of lost digits reveal different condition numbers of the problems. Among the four formulations, Formulations 1, 3 and 4 can all be viewed as stable, while Formulation 2 is unstable, introducing artificial ill-conditioning into the integral equation. 

Another interesting observation is that Formulation 4, as well as being the only one that is constructed in a black-box fashion, also gives the smallest $\kappa_2(T(x))$ and $\kappa_2(\bar{P})$. For this specific choice of boundary conditon, $T(x)$ happens to be orghogonal, rendering $\max_{x\in[a,c]} \kappa_2(T(x))=1$.
In the general case, since it is the product of orthogonal matrices and one diagonal matrix whose entries do not differ by too much, $\kappa_2(T(x))=O(1)$ is guaranteed. On the other hand, the numerical experiments suggest that the conditioning of the original problem is amplified by a factor of $O(\kappa_2(T(x)))$ in the integral formulation, explaining the success of Formulation 4 in the general case.

A complete analysis of the conditioning is quite involved and is postponed to a later date. It should answer the following questions:
\begin{itemize}
    \item How to quantify the conditioning of the original boundary value problem \eqref{eqn:inhomo}, \eqref{bc:inhomo}? What about the integral equation \eqref{eqn:eqnsigma}?
    \item When a linear transform $T(x)$ is introduced (as in \ref{thm:tx}), how does the conditioning of the problem change with respect to $T(x)$? Is $\kappa_2(T(x))=O(1)$ enough to guarantee that the conditioning remains more or less the same?
    \item If the equation is degenerate (for example, Bessel's equation), do the conclusions remain true? 
\end{itemize}

\section{Numerical examples}
In this section, we test the performance of our algorithm with four boundary value problems. All programs are implemented in FORTRAN and run on a Mac desktop with Apple M1 chip. 

Example 1 is taken from \cite{leegreengard}, which is a second-order scalar equation. We introduce a standard change of variable to reduce it to a two-dimensional boundary value system before solving it with our algorithm. It has a degenerate boundary condition $i.e. \det(A+C)=0$, and Algorithm \ref{alg:tx} is applied to construct the linear transform in a black-box fashion.
In \cite{leegreengard}, the authors demonstrate the efficiency of an automatically adaptive solver. The same strategy can be easily generalized to the system's case, but is beyond the scope of this paper. We refer readers to \cite{zhang2025} for discussions of adaptive mesh refinement. 

Examples 2-4 are high-order boundary value problems and all have degenerate boundary conditions. As before, Algorithm \ref{alg:tx} is applied to automatically address this issue. In examples 2 and 3, we show that full accuracy close to machine precision can be obtained. Example 4 is taken from \cite{starrrokhlin}, where the authors point out that the problem is stiff, causing a loss of digits. In our experiment, we observe a similar issue but can obtain 1-2 more digits than the results in \cite{starrrokhlin} when converged.

\textbf{Example 1.}
We consider a second-order scalar boundary value problem discussed in \cite{leegreengard}, which is called the viscous shock equation:
\begin{equation}
    \left\{
		\begin{aligned}
		  &\epsilon u''(x)+2x u'(x)=0,\\
	   	&u(-1)=-1,\ u(1)=1\;,
		\end{aligned}
    \right.
    \label{eqn:viscousshock}
\end{equation}

Where $\epsilon=10^{-5}$. As usual, we introduce the standard change of variable
\begin{equation} \label{eqn:viscousPhi}
\Phi(x) = 
\left[
\begin{aligned}
u(x) \\
u'(x)
\end{aligned}
\right]
\end{equation}
to reduce it to a two-dimensional system:

\begin{equation}\label{eqn:systemviscous}
    \left\{
    \begin{aligned}
        &\Phi'+\begin{bmatrix}
            0\ &-1\\
            0&\frac{2x}{\epsilon}
        \end{bmatrix}
        \Phi = 
        \begin{bmatrix}
            0\\
            0
        \end{bmatrix},\\
        &\begin{bmatrix}
            1 &0\\
            0 &0
        \end{bmatrix}
        \Phi(-1) + 
        \begin{bmatrix}
            0 &0\\
            1 &0
        \end{bmatrix}
        \Phi(1) =
        \begin{bmatrix}
            -1\\
            1
        \end{bmatrix}\;.
    \end{aligned}
    \right.
\end{equation}
Equation \eqref{eqn:viscousshock} has an exact solution given by:
\begin{equation}
    u(x)=\erf^{-1}\left(\frac{1}{\sqrt{\epsilon}}\right)\erf\left(\frac{x}{\sqrt{\epsilon}}\right)\;.
\end{equation}

We use Algorithm \ref{alg:tx} to transform equation \eqref{eqn:systemviscous} so that the boundary condition is nondegenerate. The simplest Green's function \eqref{eqn:greensfunc1} is then used to form the integral equation. 
We run our program on an adaptive grid that is exponentially clustered towards $x=0$. That is $[-1,1]=\cup_{i=0}^M ([b_i,b_{i+1}]\cup [-b_{i+1}, -b_i])$, where $b_0=0$ and $b_i=\left(\frac{1}{2}\right)^{M-i}$, $i=1, 2, \cdots, M$.
On each subinterval $[b_i, b_{i+1}]$ and $[-b_{i+1}, -b_i]$, we use $p$ Chebyshev nodes, with $p=8, 16$ respectively.

The cpu time and relative errors in $L^2$ norm are given in Table \ref{tab:example1_1} and Table \ref{tab:example1_2}. The plot of the solution along with the adaptive grid is given in Figure \ref{fig:visc}.

\begin{table}
    \centering
    \begin{tabular}{ccc}
        \hline
        subintervals & time (seconds) & relative $\mathcal{L}^2$ error \\ 
        \hline
        $2$ & $1.85\times10^{-2}$ & $2.95\times 10^{0}$ \\     
        $6$ & $2.66\times10^{-2}$& $1.33\times 10^{0}$ \\
        $10$ & $3.44\times10^{-2}$& $9.09\times 10^{-1}$ \\
        $14$ & $3.03\times10^{-2}$& $1.96\times 10^{-3}$ \\
        $18$ & $4.85\times10^{-2}$& $5.59\times 10^{-7}$ \\
        \hline
    \end{tabular}
    \caption{Numerical performance of Example 1 with $p=8$. The first column records the number of subintervals $2M$, and the second column records CPU time consumed. The  $\mathcal{L}^2$ norms of the relative error are shown in the last column.} 
    \label{tab:example1_1}
\end{table}

\begin{table}
    \centering
    \begin{tabular}{ccc}
        \hline
        subintervals & time (seconds) & relative $\mathcal{L}^2$ error \\ 
        \hline
        $2$ & $2.98\times10^{-2}$ & $6.62\times10^{0}$ \\     
        $6$ & $4.47\times10^{-2}$& $1.62\times 10^{0}$ \\
        $10$ & $5.88\times10^{-2}$& $8.84\times 10^{-2}$ \\
        $14$ & $6.74\times10^{-2}$& $2.65\times 10^{-6}$ \\
        $18$ & $8.15\times10^{-2}$& $3.37\times 10^{-12}$ \\
        \hline
    \end{tabular}
    \caption{Numerical performance of Example 1 with $p=16$. The first column records the number of subintervals $2M$, and the second column records CPU time consumed. The  $\mathcal{L}^2$ norms of the relative error are shown in the last column.} 
    \label{tab:example1_2}
\end{table}

\begin{figure}[htbp]
\centering
\includegraphics[width=0.7\textwidth]{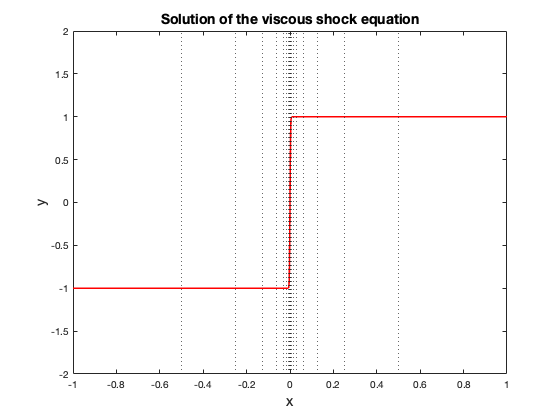}
\caption{Numerical solution of the viscous shock equation with $\epsilon=10^{-5}$. Dotted vertical lines represent endpoints of the subintervals}
\label{fig:visc}
\end{figure}

\textbf{Example 2.}
This example is taken from \cite{QAYYUM2023315}, which has a degenerate boundary condition. It demonstrates both the effectiveness of Algorithm \ref{alg:tx} and the performance of the solver for problems in relatively large dimensions.

Let us consider a $7^{th}$ order equation on the interval $[0,1]$:
\begin{equation}
    \begin{aligned}
        \phi^{(7)}(x)=\phi(x)-e^x\cdot(35+12\cdot x+2\cdot x^2)
    \end{aligned}
\end{equation}
with boundary condition
\begin{equation}
    \begin{aligned}
        &\phi(0)=0,\quad \phi^{\prime}(0)=1,\quad \phi^{\prime\prime}(0)=0,\quad \phi^{\prime\prime\prime}(0)=-3\\
        &\phi(1)=0,\quad \phi^{\prime}(1)=-e,\quad\phi^{\prime\prime}=-4\cdot e.
    \end{aligned}
\end{equation}
The exact solution is $\phi(x)=x(1-x)e^x$. We introduce the change of variable
$$(\phi_1,\phi_2,\cdots,\phi_7)=(\phi,\phi^{\prime},\cdots,\phi^{(7)})\;,$$
and solve the resulting 7-dimensional boundary value system. In this case, the transformation matrix $T(x)$ is orthogonal. Regarding the computational grid, we start with one interval $[a,c]$ and successively subdivide each interval into two halves, so that with $i$ subdivisions, we have $2^i$ subintervals on the finest level of refinement. We test the solver with $p=6,8,12,16$ and set the error tolerance to $10^{-4}$ to $10^{-12}$. The intervals are subdivided until the error tolerance is reached for each $p$. 

The CPU time consumed for each case is presented in Figure \ref{fig: 7order}. The result demonstrates that the algorithm is fast and accurate even if the dimension is relatively large. In this case, a reasonable choice of the number of Chebyshev nodes $p$ is $8\leq p\leq 16$, and generally speaking, a larger $p$ is preferred for a higher precision requirement. This example also sets the framework for an automatically adaptive method. 

\begin{figure}[htbp]
\centering
\includegraphics[width=0.7\textwidth]{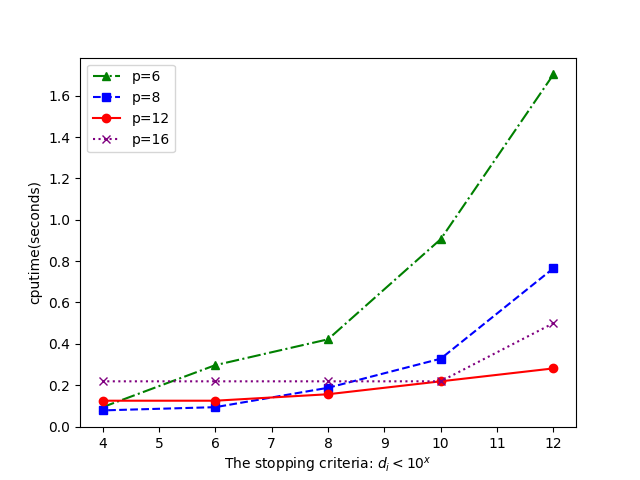}
\caption{Example 2. CPU time consumed for each $p$ and error tolerance.}
\label{fig: 7order}
\end{figure}

\textbf{Example 3.}
We consider the following $7^{th}$ order equation on the interval $[2,10]$ (also taken from \cite{QAYYUM2023315}):
\begin{equation}
    \begin{aligned}
        \phi^{(7)}(x)=x\cdot\phi(x)+e^x\cdot(-6-2\cdot x+x^2)
    \end{aligned}
\end{equation}
with boundary condition
\begin{equation}
    \begin{aligned}
        &\phi(0)=1,\quad \phi^{\prime}(0)=0,\quad \phi^{\prime\prime}(0)=-1,\quad \phi^{\prime\prime\prime}(0)=-2\\
        &\phi(1)=-9\cdot e^{10},\quad \phi^{\prime}(1)=-10\cdot e^{10},\quad\phi^{\prime\prime}=-11\cdot e^{10}.
    \end{aligned}
\end{equation}
This equation has an exact solution $\hat\phi(x)=(1-x)e^x$. As we did in the previous example, we convert this scalar equation to a 7-dimensional system and solve it with our fast solver. In this case, $p=8$ is fixed. The same computational grid as in the previous example is adopted.
After $i$ steps of subdivision, the relative error in $\mathcal{L}^2$ norm, i.e. $\frac{\|\hat{\phi}(x)-\phi_i(x)\|_2}{\|\hat{\phi}(x)+\phi_i(x)\|_2}$, is presented in Table $\ref{tab:example3}$. With the finest refinement, a relative $\mathcal{L}^2$ error of magnitude $10^{-15}$ is achieved.

\begin{table}
    \centering
    \begin{tabular}{ccc}
        \hline
        subintervals & time (seconds) & relative $\mathcal{L}^2$ error \\ 
        \hline
        $7$ & $6.25\times10^{-2}$ & $4.54\times 10^{-7}$ \\        
        $15$ & $1.41\times10^{-1}$& $3.50\times 10^{-9}$ \\
        $31$ & $1.72\times10^{-1}$& $1.72\times 10^{-11}$ \\
        $63$ & $3.59\times10^{-1}$& $3.83\times 10^{-13}$ \\
        $127$ & $5.47\times10^{-1}$& $1.89\times 10^{-15}$ \\
        \hline
    \end{tabular}
    \caption{Numerical performance of Example 3 with $p=8$. The first column records the number of subintervals, and the second column records CPU time consumed. The $\mathcal{L}^2$ norms of the relative error are shown in the last column.}
    \label{tab:example3}
\end{table}

\textbf{Example 4.}
The last example can be found in \cite{salvadori1961numerical} and \cite{starrrokhlin}. It is an equation of $4^{th}$ order in the Euler-Bernoulli beam theory, which demonstrates the deflection of a beam. The equation is given by:
\begin{equation}
    y^{\prime\prime\prime\prime}(x)+\frac{k}{E\cdot I}\cdot y(x)=\frac{q}{E\cdot I}
\end{equation}
subject to the boundary condition
\begin{equation}
    y(0)=y^{\prime}(0)=y(L)=y^{\prime\prime}(L)=0.
\end{equation}
Here, $EI$ is the Flexural Rigidity of the beam, $q$ is the load acting transversely on the beam, and $k$ is the  Winkler constant representing force per unit deflection per unit length of the beam. The values used in our example are as follows:
\begin{equation}
\begin{aligned}
    &L = 1.2\times 10^2,\quad E=3.0\times 10^7,\quad I=3.0\times 10^3,\\
    &q = 4.34\times 10^4,\quad k=2.604\times 10^3. 
\end{aligned}
\end{equation}
In this case, the exact solution $\hat{\phi}(x)$ is unknown, we
carry out a self convergence study, 
approximating it by
$\frac{\|\hat{\phi}_i(x)-\phi_i(x)\|_2}{\|\hat{\phi}_i(x)+\phi_i(x)\|_2}$,
where $\hat{\phi}_i(x)$ is an accurate enough numerical solution obtained by using twice as many intervals as the finest one,
which serves as an approximation of the exact solution. The result is shown in Table $\ref{tab:example4}$
\begin{table}
    \centering
    \begin{tabular}{ccc}
        \hline
        subintervals & time (seconds) & relative $\mathcal{L}^2$ error \\ 
        \hline
        $7$ & $3.906\times10^{-2}$ & $3.448\times 10^{-6}$ \\        
        $15$ & $4.688\times10^{-2}$& $1.983\times 10^{-8}$ \\
        $31$ & $1.250\times10^{-1}$& $2.514\times 10^{-10}$ \\
        $63$ & $2.344\times10^{-1}$& $2.012\times 10^{-10}$ \\
        $127$ & $3.281\times10^{-1}$& $1.759\times 10^{-10}$ \\        
        \hline
    \end{tabular}
    \caption{Numerical performance of Example 4 with $p=8$. The first column records the number of subintervals, and the second column records CPU time consumed. The $\mathcal{L}^2$ norms of the relative error are shown in the last column.}
    \label{tab:example4}
\end{table}

The relative $\mathcal{L}^2$ error stagnates at the magnitude of $10^{-10}$, which is relatively larger compared to other examples.
As is pointed out in \cite{starrrokhlin}, this is caused by the stiffness/ill-conditioning of the problem itself. We observe that on the same order $p=8$, we obtain one to two more digits than the results presented in \cite{starrrokhlin}, in addition to providing a black-box algorithm to deal with the degenerate boundary condition.

\section{Conclusions}\label{sec:conclusion}

In this paper, we have studied the integral equation methods for linear systems of two-point boundary value problems in detail, focusing on the mathematical formulation of the integral equation and the conditioning property of the resulting equation. Two different approaches are presented for the so-called degenerate boundary conditions which persist in actual applications. Although a complete analysis of conditioning remains an open problem, a numerical study shows that both approaches can be made stable in practice. Among them, we recommend the second approach, which is properly black-boxed, for daily use.

As a result, nice properties of second-kind Fredholm integral
equations are preserved, which sets the foundation for a black-box solver that is fast, accurate, and automatically adaptive. Further details of the numerical components will be presented in a subsequent paper~\cite{zhang2025}, which will also extend the existing algorithm to nonlinear equations.

\section*{Acknowledgements}
The authors were supported in part by National Key Research and Development Program of China under grant 2023YFA1008902 and in part by National Natural Science Foundation of China under grant 12301515. We would like to thank Professors Leslie Greengard, Alex Barnett, Charles Epstein, and Manas Rachh for many helpful discussions.

\bibliographystyle{plain}
\bibliography{bvplin}
\end{document}